\documentclass[12pt]{amsart}
\usepackage[a4paper,margin=1in]{geometry}
\usepackage{amsmath,amssymb,amsthm,mathtools}
\usepackage{bm}
\usepackage{enumitem}
\usepackage{microtype}
\usepackage{hyperref}
\usepackage{tikz}
\usepackage{pgfplots}
\pgfplotsset{compat=1.18}
\usetikzlibrary{calc,positioning,arrows.meta,patterns,decorations.pathmorphing}
\usepackage{subcaption}

\title
[]
{Inverse source problems with reduced interior data for a coupled reaction-diffusion system}

\author{Xinyue Luo$^{1}$, Masahiro Yamamoto$^{2}$, Jin Cheng$^{1}$}

\theoremstyle{plain}
\newtheorem{theorem}{Theorem}[section]
\newtheorem{lemma}{Lemma}[section]

\newtheorem{corollary}{Corollary}[section]
\theoremstyle{definition}
\newtheorem{definition}{Definition}[section]
\theoremstyle{remark}
\newtheorem{remark}{Remark}[section]

\newcommand{\R}{\mathbb{R}}

\newcommand{\p}{\partial}
\newcommand{\ppp}{\partial}
\newcommand{\ooo}{\overline}
\newcommand{\OOO}{\Omega}
\newcommand{\www}{\widetilde}

\newcommand{\norm}[1]{\left\lVert#1\right\rVert}

\newcommand{\Om}{\Omega}

\newcommand{\Dt}{\p_t}

\newcommand{\Lap}{\Delta}

\usepackage{xcolor}
\newcommand{\revi}[1]{{\color{blue}#1}}

\begin{document}

\maketitle

\begin{center}
\footnotesize
$^{1}$School of Mathematical Sciences, Fudan University, Shanghai, China\\
$^{2}$Department of Mathematical Sciences, The University of Tokyo, Tokyo, Japan\\[0.3em]
E-mail: xinyueluo21@m.fudan.edu.cn, myama@next.odn.ne.jp, jcheng@fudan.edu.cn
\end{center}
\begin{abstract}
We consider a two-component semilinear reaction-diffusion system in a bounded spatial domain $\Omega$ over a time interval $(0,T)$, which governs the water density $u(x,t)$ and the vegetation biomass density $v(x,t)$ for $x\in\Omega$ and $0<t<T$. In this system, called the Klausmeier-Gray-Scott model, we assume that an unknown source depends only on the spatial variable and appears in the reaction-diffusion equation for $u$.

The main subject is the inverse source problem of determining a source term from limited data on $(u,v)$.
We establish two kinds of stability estimates by means of Carleman estimates. First, a Carleman estimate with a singular weight yields a Lipschitz stability estimate for the inverse source problem from data consisting of a snapshot $u(\cdot,t_0)$ in $\Omega$ and $(u,v)$ in a subdomain $\omega$ over a time interval. Second, without assuming boundary data, we prove a H\"older stability estimate in any interior subdomain $\Omega_0$ satisfying $\overline{\Omega_0}\subset\Omega$. We further study how much the observation data can be reduced while preserving uniqueness and stability in the inverse problem under suitable additional conditions.
\end{abstract}

\section{Introduction}\label{sec:intro}

The spontaneous formation of spatial patterns in dryland vegetation is a
striking ecological phenomenon observed across semiarid and arid regions
worldwide.
Regular structures---stripes, spots, gaps, and labyrinths---emerge from the
interaction between water availability and plant growth, and serve as
indicators of ecosystem health and potential desertification
\cite{rietkerk2008regular}.
Among the mathematical frameworks proposed to explain such self-organisation,
the model introduced by Klausmeier \cite{klausmeier1999regular} has been
particularly influential.

In the Klausmeier-Gray-Scott model, the interaction between water and vegetation is governed by a facilitative feedback mechanism: denser vegetation enhances water infiltration and uptake, which in turn promotes further plant growth. This positive feedback, combined with water loss through evaporation and plant mortality, drives the self-organised pattern formation.

We now formulate the mathematical model. Let $\Om\subset\R^n$, $n\ge 1$, be a bounded domain with $C^2$ boundary $\p\Om$, and let $T>0$. We write $Q:=\Om\times(0,T)$ and $\Sigma:=\p\Om\times(0,T)$. The Klausmeier-Gray-Scott system couples a water density $u$ with a vegetation biomass density $v$ through the nonlinear interaction term $uv^2$:
\begin{equation}\label{eq:forward}
\begin{cases}
\Dt u = d_1 \Lap u - u v^2 - a_1(x,t)\,u + f(x), & (x,t)\in Q,\\[4pt]
\Dt v = d_2 \Lap v + u v^2 - a_2(x,t)\,v, & (x,t)\in Q,
\end{cases}
\end{equation}
with Dirichlet boundary conditions
\begin{equation}\label{eq:bc}
u=g,\qquad v=h\qquad\text{on }\Sigma,
\end{equation}
and initial conditions
\begin{equation}\label{eq:ic}
u(\cdot,0)=u_0,\qquad v(\cdot,0)=v_0\qquad\text{in }\Om.
\end{equation}
Here $d_1,d_2>0$ are diffusion constants,
$a_1,a_2\in W^{1,\infty}(Q)$ are known nonnegative reaction coefficients 
representing loss rates, that is, evaporation for water, and mortality for 
vegetation. 
Moreover, $f\in L^2(\Om)$ is an unknown spatially dependent source and 
represents, for instance, external water supply.

The term $-uv^2$ in the first equation models water consumption by vegetation:
the uptake rate is proportional to the water density and the square of the
biomass, reflecting the facilitative effect of dense vegetation on water
infiltration.
The corresponding term $+uv^2$ in the second equation describes the conversion
of absorbed water into biomass growth.

In practice, the spatial distribution of water input into an ecosystem---due to
rainfall heterogeneity, terrain runoff, or irrigation---is often unknown or
difficult to measure directly.
This motivates the study of inverse source problems, in which one seeks to recover a spatially varying source $f(x)$ 
from partial observations of the solution.
Patterns of solutions and their dependence on system parameters have been 
studied extensively from the perspective of the forward problem
\cite{consolo2019secondary, sherratt2010pattern, ursino2007modeling}, but 
to the best of the authors' knowledge, there are no publications on 
inverse problems for Klausmeier-Gray-Scott system.
The present work addresses the mathematical analysis for the source
identifiability from limited measurements.

A further motivation arises from the practical constraints of ecological
monitoring.
Field measurements of vegetation biomass are often obtained through remote
sensing over large spatial regions, while water content measurement data 
are usually available only at sparse sensor locations.
This asymmetry between the observability of the two components leads naturally
to the question of whether observations of fewer state variables
suffice for source recovery, that is, the data-reduction problem is also 
studied in this paper.

Inverse problems for parabolic equations are usually ill-posed and require
overdetermined data for uniqueness and stability; see
\cite{isakov1998inverse,Y25} for general background.
A powerful approach to quantitative stability relies on Carleman estimates, following the method introduced by Bukhgeim and Klibanov \cite{bukhgeim1981global}; see also
Klibanov \cite{klibanov1992inverse}.
For scalar parabolic equations, Lipschitz stability for inverse source problems
was established by Imanuvilov and Yamamoto \cite{imanuvilov1998lipschitz}, 
Yamamoto and Zou \cite{YZ}; see also \cite{huang2020stability, yamamoto2009carleman}.
Inverse coefficient problems with initial or final time data were studied in
\cite{ImanuvilovYamamoto2024}.

For inverse problems for partial differential equations, 
a fundamental difficulty is that measurements may be available only for a 
subset of the state components or on a restricted interior region.
The seminal work of Cristofol, Gaitan, and Ramoul
\cite{cristofol2006inverse} established stability for a $2\times 2$ 
reaction--diffusion system using a Carleman estimate with observation of only one component; see also \cite{benabdallah2009inverse} for an extension by
Benabdallah, Cristofol, Gaitan, and Yamamoto to systems with first-order coupling.
Identification of two coefficients from data of one component in a nonlinear system was studied in \cite{cristofol2012identification}.
Further contributions on inverse problems for reaction--diffusion systems 
include \cite{boulakia2021numerical, kaltenbacher2020inverse, pyatkov2017some}.
Despite their significance in applications, inverse source problems for the Klausmeier--Gray--Scott vegetation model and related ecological and chemical reaction--diffusion systems have not been sufficiently investigated.

Now we formulate our main inverse problems for \revi{\eqref{eq:forward}}.
Let $\omega \subset\Om$ be a fixed nonempty open subdomain satisfying $\ooo{\omega} \subset \Om$ and let $t_0\in(0,T)$ be a fixed observation time. 
Here and henceforth, $\ooo{A}$ denotes the closure of a set $A$ under 
consideration.
Consider two admissible triples $(u,v,f)$ and $(\tilde u,\tilde v,\tilde f)$ in $\mathcal{U}(M,M_0)$ that solve the forward system \eqref{eq:forward}--\eqref{eq:bc} with the same coefficients $a_1, a_2$ and
boundary data $g, h$.

Our principal interest is 
\\
\medskip
\noindent\textbf{Inverse Source Problem.}
\emph{Given limited observation data of the solution $(u,v)$ to 
\eqref{eq:forward}, determine the unknown spatial source $f\in L^2(\Omega)$.}
\\

More precisely, our interest is in establishing stability estimates for the
source difference $f - \tilde{f}$ in terms of discrepancies in the observed
data. For determining $f$, we adopt the following kinds of data.
We arbitrarily fix $t_1, t_2 \in (0,T)$ with $t_1 < t_2$ and set 
\begin{equation}
 I:= (t_1, t_2).
 \label{eq:4}
\end{equation}
We discuss the inverse problem by the following four kinds of data.
\begin{enumerate}[label=$(\mathcal{M}_{\arabic*})$, leftmargin=3.5em]
\item\label{M1} \textbf{Reference measurement:}
$u(\cdot,t_0)|_\Omega$, $u|_{\omega\times I}$, $v|_{\omega\times I}$.
\item\label{M2} \textbf{Reduced measurement (observation of $v$ only on
$\omega\times I$):}
$u(\cdot,t_0)|_\Omega$, $v|_{\omega\times I}$.
\item\label{M3} \textbf{Reduced measurement (observation of $u$ only on
$\omega\times I$, with $f|_\omega$ given):}
$u(\cdot,t_0)|_\Omega$, $u|_{\omega\times I}$, and $f|_\omega$ is given
a priori.
\item\label{M4} \textbf{Reduced measurement (two snapshots with observation
of $u$ on $\omega\times I$):}
$u(\cdot,t_0)|_\Omega$, $v(\cdot,t_0)|_\Omega$, $u|_{\omega\times I}$.
\end{enumerate}

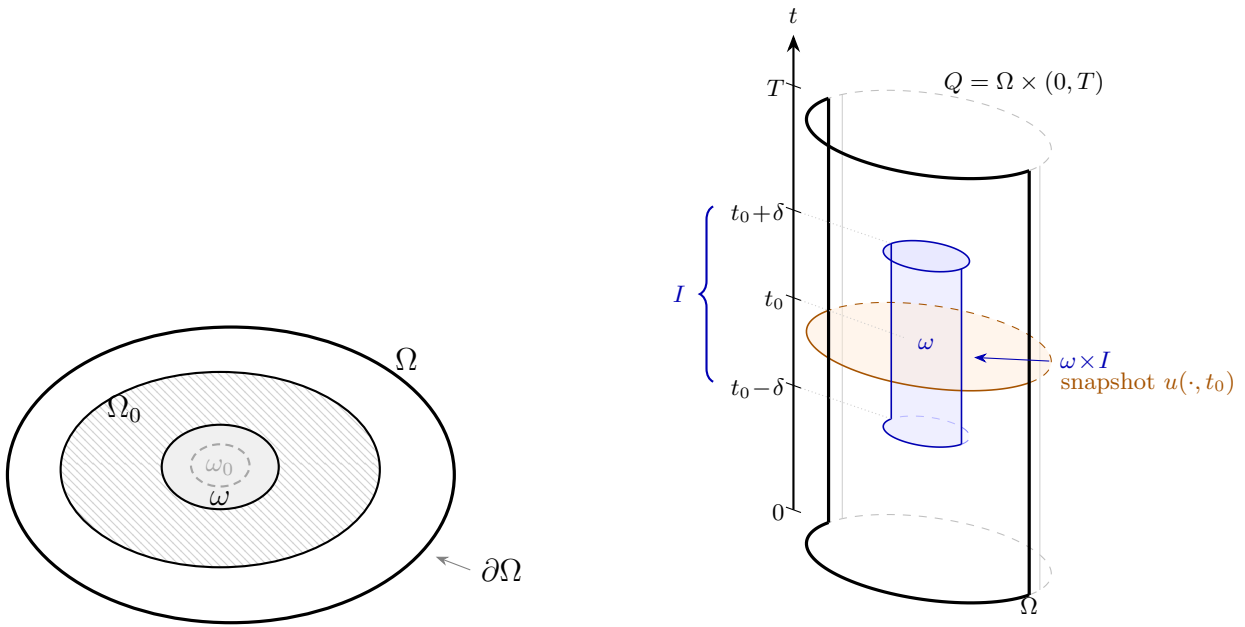
\begin{figure}[htbp]
\centering
\begin{subfigure}[t]{0.45\textwidth}
\centering
\begin{tikzpicture}[>=Stealth,thick,
 every node/.style={font=\normalsize},
 scale=0.7]

 \draw[very thick]
  (0,0) ellipse (4.2cm and 2.8cm);
 \node at (3.3,2.2) {$\Omega$};

 \draw[->,thin,gray] (4.5,-1.8) node[right,black]{$\partial\Omega$}
  -- (3.85,-1.55);

 \begin{scope}
  \clip (-0.2,0.1) ellipse (3.0cm and 1.85cm);
  \fill[pattern=north west lines, pattern color=gray!35]
  (-0.2,0.1) ellipse (3.0cm and 1.85cm);
 \end{scope}
 \begin{scope}
  \clip (-0.2,0.15) ellipse (1.1cm and 0.8cm);
  \fill[white] (-0.2,0.15) ellipse (1.1cm and 0.8cm);
 \end{scope}
 \draw[thick] (-0.2,0.1) ellipse (3.0cm and 1.85cm);
 \node at (-2.0,1.25) {$\Omega_0$};

 \fill[gray!12] (-0.2,0.15) ellipse (1.1cm and 0.8cm);
 \draw[thick] (-0.2,0.15) ellipse (1.1cm and 0.8cm);
 \node at (-0.2,-0.45) {$\omega$};

 \draw[densely dashed, gray!70]
  (-0.2,0.18) ellipse (0.55cm and 0.4cm);
 \node[gray!70,font=\footnotesize] at (-0.2,0.15) {$\omega_0$};


\end{tikzpicture}
\caption{Nested subdomains satisfying $\omega_0\subset \overline{\omega_0}
\subset \omega \subset \overline{\omega}\subset \Omega_0 
\subset \overline{\Omega_0} \subset\Omega$.}
\label{fig:domain}
\end{subfigure}
\hfill
\begin{subfigure}[t]{0.45\textwidth}
\centering
\begin{tikzpicture}[>=Stealth,thick,
  x={(0.88cm,-0.32cm)}, y={(0.88cm,0.32cm)}, z={(0cm,1cm)},
  scale=0.75]

 \def\rx{2.0} \def\ry{1.4}
 \def\H{7.5} 
 \def\tA{2.2} \def\tB{5.3} \def\tZ{3.75}
 \def\ox{-0.1} \def\oy{0.05}
 \def\orx{0.7} \def\ory{0.5}

 \draw[gray!50,thin,dashed]
  plot[domain=0:180,samples=60]
  ({\rx*cos(\x)},{\ry*sin(\x)},0);

 \draw[gray!40,very thin]
  ({\rx*cos(170)},{\ry*sin(170)},0)
  -- ({\rx*cos(170)},{\ry*sin(170)},\H);
 \draw[gray!40,very thin]
  ({\rx*cos(10)},{\ry*sin(10)},0)
  -- ({\rx*cos(10)},{\ry*sin(10)},\H);

 \draw[blue!50,thin,dashed]
  plot[domain=0:180,samples=40]
  ({\ox+\orx*cos(\x)},{\oy+\ory*sin(\x)},\tA);

 \fill[blue!12,opacity=0.50]
  plot[domain=180:360,samples=40]
  ({\ox+\orx*cos(\x)},{\oy+\ory*sin(\x)},\tA)
  -- plot[domain=360:180,samples=40]
  ({\ox+\orx*cos(\x)},{\oy+\ory*sin(\x)},\tB)
  -- cycle;
 \fill[blue!6,opacity=0.30]
  plot[domain=0:180,samples=40]
  ({\ox+\orx*cos(\x)},{\oy+\ory*sin(\x)},\tA)
  -- plot[domain=180:0,samples=40]
  ({\ox+\orx*cos(\x)},{\oy+\ory*sin(\x)},\tB)
  -- cycle;

 \fill[orange!15,opacity=0.50]
  plot[domain=0:360,samples=80]
  ({\rx*cos(\x)},{\ry*sin(\x)},\tZ)
  -- cycle;
 \draw[orange!65!black,thin,dashed]
  plot[domain=0:180,samples=60]
  ({\rx*cos(\x)},{\ry*sin(\x)},\tZ);
 \draw[orange!65!black,semithick]
  plot[domain=180:360,samples=60]
  ({\rx*cos(\x)},{\ry*sin(\x)},\tZ);
 \node[orange!65!black,font=\scriptsize,right]
  at ({\rx*1.1},0.2,\tZ)
  {snapshot $u(\cdot,t_0)$};

 \draw[blue!70!black,semithick]
  plot[domain=180:360,samples=40]
  ({\ox+\orx*cos(\x)},{\oy+\ory*sin(\x)},\tA);
 \fill[blue!18,opacity=0.40]
  plot[domain=0:360,samples=60]
  ({\ox+\orx*cos(\x)},{\oy+\ory*sin(\x)},\tB)
  -- cycle;
 \draw[blue!70!black,semithick]
  plot[domain=0:360,samples=60]
  ({\ox+\orx*cos(\x)},{\oy+\ory*sin(\x)},\tB);
 \draw[blue!70!black,semithick]
  ({\ox-\orx},{\oy},\tA) -- ({\ox-\orx},{\oy},\tB);
 \draw[blue!70!black,semithick]
  ({\ox+\orx},{\oy},\tA) -- ({\ox+\orx},{\oy},\tB);
 \node[blue!70!black,font=\scriptsize]
  at ({\ox},{\oy},{0.5*\tA+0.5*\tB}) {$\omega$};

 \draw[very thick]
  plot[domain=180:360,samples=60]
  ({\rx*cos(\x)},{\ry*sin(\x)},0);
 \draw[very thick] (-\rx,0,0)--(-\rx,0,\H);
 \draw[very thick] (\rx,0,0)--(\rx,0,\H);
 \draw[very thick]
  plot[domain=180:360,samples=60]
  ({\rx*cos(\x)},{\ry*sin(\x)},\H);
 \draw[gray!50,thin,dashed]
  plot[domain=0:180,samples=60]
  ({\rx*cos(\x)},{\ry*sin(\x)},\H);

 \draw[->,thick]
  (-\rx-0.7,0,0) -- (-\rx-0.7,0,{\H+0.9})
  node[above,font=\scriptsize]{$t$};
 \draw[thin] (-\rx-0.85,0,0)--(-\rx-0.55,0,0)
  node[left=2pt,font=\scriptsize]{$0$};
 \draw[thin] (-\rx-0.85,0,\tA)--(-\rx-0.55,0,\tA)
  node[left=2pt,font=\scriptsize]{$t_0\!-\!\delta$};
 \draw[thin] (-\rx-0.85,0,\tZ)--(-\rx-0.55,0,\tZ)
  node[left=2pt,font=\scriptsize]{$t_0$};
 \draw[thin] (-\rx-0.85,0,\tB)--(-\rx-0.55,0,\tB)
  node[left=2pt,font=\scriptsize]{$t_0\!+\!\delta$};
 \draw[thin] (-\rx-0.85,0,\H)--(-\rx-0.55,0,\H)
  node[left=2pt,font=\scriptsize]{$T$};

 \draw[decorate,
  decoration={brace,amplitude=5pt},
  blue!70!black,thick]
  (-\rx-1.6,-0.7,\tA) -- (-\rx-1.6,-0.7,\tB)
  node[midway,left=6pt,blue!70!black,font=\scriptsize]{$I$};

 \draw[gray!40,densely dotted,thin]
  (-\rx-0.55,0,\tA) -- ({\ox-\orx},{\oy},\tA);
 \draw[gray!40,densely dotted,thin]
  (-\rx-0.55,0,\tB) -- ({\ox-\orx},{\oy},\tB);
 \draw[gray!40,densely dotted,thin]
  (-\rx-0.55,0,\tZ) -- (-0.5,0,\tZ);

 \node[font=\scriptsize] at (0,{\ry+0.5},{\H+0.3})
  {$Q=\Omega\times(0,T)$};
 \node[font=\scriptsize] at (\rx+0.3,-0.3,0) {$\Omega$};
 \draw[<-,blue!70!black,thin]
  ({\ox+\orx+0.15},{\oy+0.1},{0.5*\tA+0.5*\tB})
  -- ({1.6},{0.8},{0.5*\tA+0.5*\tB})
  node[right,font=\scriptsize]{$\omega\!\times\! I$};

\end{tikzpicture}
\caption{Observation subcylinder $\omega\times I$ and snapshot $u(\cdot,t_0)|_\Omega$ in $Q$.}
\label{fig:cylinder}
\end{subfigure}

\caption{Domain geometry and measurements.}
\label{fig:domain-cylinder}
\end{figure}

Throughout all the cases \ref{M1} - \ref{M4}, we emphasize that data
$u(\cdot,t_0)|_\Om$ is indispensable for the recovery of $f$, because
the source $f(x)$ appears in the first equation
of \eqref{eq:forward}, and extracting $f$ from the equation at $t=t_0$
requires controlling $\p_t u$, $\Lap u$, and $uv^2$ at the observation time.
The second equation contains no unknown source; consequently, data 
$v(\cdot,t_0)|_\Om$ alone are insufficient unless combined with
additional data.
\\

This paper is organised as follows.
Section~\ref{sec:results} introduces the admissible class and states the main results.
Theorem~\ref{thm:reference} establishes the Lipschitz stability estimate, and
Theorem~\ref{thm:holder} gives the H\"older stability estimate.
The reduced-data results are stated as Corollaries~\ref{cor:red2}
--\ref{cor:red4}.
Section~\ref{sec:proof-lipschitz} is devoted to the proof of
Theorem~\ref{thm:reference} by means of a Carleman estimate with a singular weight.
Section~\ref{sec:proof-holder} proves Theorem~\ref{thm:holder} by using a
Carleman estimate with a regular weight.
Section 5 provides the proofs of Corollaries~\ref{cor:red2}--\ref{cor:red4}.
Section~\ref{sec:conclusion} concludes the paper with some further discussion.

\section{Main results}\label{sec:results}

For the statements of our results, we need to introduce function spaces.
First, we use the anisotropic Sobolev space
\[
H^{2,1}(Q)
:= L^2(0,T;H^2(\Omega))\cap H^1(0,T;L^2(\Omega)),
\]
equipped with the standard norm
\[
\|w\|_{H^{2,1}(Q)}
:=\|w\|_{L^2(0,T;H^2(\Omega))}+\|\partial_t w\|_{L^2(Q)}.
\]
This space is frequently used for strong solutions to parabolic equations and
systems (e.g., \cite{ladyzhenskaia1968linear,lunardi2012analytic}).
In particular, the trace mapping $w\mapsto w(\cdot,t)$ is well-defined for
$w\in H^{2,1}(Q)$, and such functions admit a representative in
$C([0,T];H^1(\Omega))$ (e.g., \cite{ladyzhenskaia1968linear}).
We also employ the space $W^{1,\infty}(Q)$ of essentially bounded functions 
with essentially bounded first-order derivatives.
Whenever pointwise expressions such as $uv^2$ are used, they are interpreted
through the a priori bounds in the admissible class introduced later.

We introduce the a priori class in which the inverse problem is studied.

\begin{definition}[Admissible class of solutions $(u,v)$]\label{def:U}
We arbitrarily fix a constant $M>0$.
We denote by $\mathcal U=\mathcal U(M)$ the set of triples $(u,v,f)$
such that:
\begin{enumerate}[label=(\roman*), leftmargin=2.5em]
\item $(u,v)\in H^{2,1}(Q)^2$ solves \eqref{eq:forward}--\eqref{eq:bc} in $Q$,
\item 
\[
u,v\in H^{2,1}(Q)\cap W^{1,\infty}(Q),\qquad
\p_t u,\;\p_t v\in H^{2,1}(Q),
\]
\item the source $f\in L^2(\Om)$ satisfies the bounds
\[
\norm{u}_{W^{1,\infty}(Q)}+\norm{v}_{W^{1,\infty}(Q)}
+\norm{f}_{L^2(\Om)}\le M.
\]
\end{enumerate}
\end{definition}

Henceforth we assume that the coefficients satisfy
\[
\sum_{j=1}^2\sum_{k=0}^1 \norm{\p_t^k a_j}_{L^\infty(Q)}\le M_0
\]
with some constant $M_0 > 0$.

We emphasize that we assume the existence of solutions $(u,v), 
(\www{u}, \www{v}) \in \mathcal{U}$.
There have been detailed results for the forward problem \eqref{eq:forward}--\eqref{eq:ic} and several methodologies have been established. See, e.g.,
\cite{amann1995linear, henry1981geometric, pierre2010global, smoller2012shock}.

By such classical results, we can describe conditions of $a_1, a_2, f$, boundary values $g, h$ and initial values $u_0, v_0$ in order that 
the solutions to the forward problem, belong to $\mathcal{U}$. However, since our main interest is the inverse problem, not pursuing the well-posedness of the forward problem, we focus on the stability within the admissible set $\mathcal{U}$.
Our method for the inverse problem can be adjusted to some extent for less 
regular class of $(u,v)$, but we do not treat other choices of the admissible
set.

The regularity assumptions in the admissible set $\mathcal{U}$ 
are necessary for the application of the Carleman estimates. Indeed, if the data $(u_0,v_0,g,h,f,a_1,a_2)$ are sufficiently smooth and satisfy the required compatibility conditions, then the corresponding solution $(u,v)$ enjoys higher regularity, in particular $\p_t u,\p_t v\in H^{2,1}(Q)$.
We refer to a standard parabolic regularity theory for the resulting linear system; see \cite[Chapter~IV, Theorem~9.1]{ladyzhenskaia1968linear}.
These assumptions can be partially relaxed, but we do not pursue this here.
\\

We are ready to state our main results for the inverse source problem of 
determining $f(x)$.


\begin{theorem}[Lipschitz stability with data $(\mathcal{M}_1)$]
\label{thm:reference}
We arbitrarily choose 
a subdomain $\omega$ satisfying $\ooo{\omega} \subset\Omega$ and 
a time interval $I:= (t_1, t_2)$ with $0<t_1<t_0<t_2<T$.
Let $(u,v,f)$ and $(\widetilde u,\widetilde v,\widetilde f)$ belong to
$\mathcal U(M)$ and solve \eqref{eq:forward}--\eqref{eq:bc} with the
same coefficients and boundary data.
There exists a constant $C>0$, depending only on
$\Omega,\omega,T,t_0,t_1,t_2,M,M_0$, such that
\begin{equation}\label{eq:stab-reference}
\|f-\widetilde f\|_{L^2(\Omega)}
\le C\Bigl(
\|u(\cdot,t_0)-\widetilde u(\cdot,t_0)\|_{H^2(\Omega)}
+\sum_{\ell=0}^1 \|\partial_t^\ell(u-\widetilde u)\|_{L^2(\omega\times I)}
+\sum_{\ell=0}^1 \|\partial_t^\ell(v-\widetilde v)\|_{L^2(\omega\times I)}
\Bigr).
\end{equation}
\end{theorem}

Theorem~\ref{thm:reference} establishes the Lipschitz stability over the whole spatial domain $\OOO$ in determining $f(x)$. A standard inverse source problem for a scalar parabolic equation of the form $\p_t w-d\Lap w+a\,w=f(x)$ requires, at minimum, a spatial snapshot $w(\cdot,t_0)|_\Om$ together with time-resolved interior observations $w|_{\omega\times I}$; see e.g., \cite{imanuvilov1998lipschitz,Y25}. For the coupled system \eqref{eq:forward}, the source $f(x)$ appears only in the first (water) equation. A natural observation data set therefore consists of: a full-domain snapshot of $u$ at the observation time $t_0$ for extracting the source via the equation, and time-resolved observations of both $u$ and $v$ on the subdomain $\omega\times I$ for controlling the coupled dynamics through the Carleman estimate.

The next question is \emph{data-reduction}:
\emph{can some of these observation data be reduced to preserve the stability?}

This question is of both theoretical and practical interest.
Theoretically, it reveals which structural features of the nonlinear coupling carry identifiability information. Practically, it addresses the common situation in ecological monitoring where observations of one component (e.g., vegetation via remote sensing) are available but measurements of the other (e.g., soil moisture) are not.

Our approach to data reduction is similar to the strategy of
\cite{benabdallah2009inverse, cristofol2006inverse}, where data of a single component suffice for determining coefficients in some linear systems with two components. We reduce the data set through the quadratic coupling term $uv^2$ of the two equations in \eqref{eq:forward}, under a positivity condition, observing that one component on $\omega\times I$ determines the other through the forward equations.
Below, we consider other data sets $(\mathcal{M}_2) - (\mathcal{M}_4)$, and state only the uniqueness results. The corresponding stability results are derived readily, but, in order to clarify the role of these reduced data relative to the data in Theorem~\ref{thm:reference}, we restrict ourselves here to the uniqueness.

\begin{corollary}[Uniqueness via $\mathcal{M}_2$: observation of $v$ only]
\label{cor:red2}
Let $(u,v,f)$ and $(\tilde u,\tilde v,\tilde f)$ belong to 
$\mathcal{U}(M)$ and solve \eqref{eq:forward}--\eqref{eq:bc} with the same coefficients $a_1, a_2 \in C(\ooo Q)$, and boundary data. Moreover, we further assume
$u \in C(\ooo\OOO \times \ooo{I})$ and 
$v \in C(\ooo\OOO \times \ooo{I}) \cap C^2(\OOO \times I)$,
and 
\begin{equation}
\label{eq:5}
v_0 \ge 0 \quad \text{in } \OOO, \quad \text{and}\quad
h \ge 0 \quad \text{on } \Sigma  
\end{equation}
and
\begin{equation}
\label{eq:6}
v_0 \not\equiv 0 \quad \text{in } \OOO. 
\end{equation}
If
$$
u(\cdot,t_0) = \www{u}(\cdot,t_0) \quad \mbox{in $\OOO$} \quad 
\mbox{and} \quad v\vert_{\omega \times I}= \www v\vert_{\omega \times I},
$$
then $f=\widetilde f$ in $\Omega$.
\end{corollary}

\begin{corollary}[Uniqueness via $\mathcal{M}_3$: observation of $u$ 
with $f|_\omega$ given]
\label{cor:red3}
In addition to Corollary 2.1, we further assume 
$u \in C(\ooo\OOO \times \ooo{I}) \cap C^2(\OOO \times I)$, 
$v \in C(\ooo\OOO \times \ooo{I})$, and $f \ge 0$ in $\OOO$, $f\vert_{\omega}
= \www f\vert_{\omega}$, and
\begin{equation}
\label{eq:7}
u_0 \ge 0 \quad \text{in } \OOO, \quad \text{and}\quad
g \ge 0 \quad \text{on } \Sigma  
\end{equation}
and
\begin{equation}
\label{eq:8}
u_0 \not\equiv 0 \quad \text{in }\OOO. 
\end{equation}
If
$$
u(\cdot,t_0) = \www{u}(\cdot,t_0) \quad \mbox{in $\OOO$} \quad 
\mbox{and} \quad u\vert_{\omega \times I}= \www u\vert_{\omega \times I},
$$
then $f=\widetilde f$ in $\Omega$.

\end{corollary}

\begin{corollary}[Uniqueness via $\mathcal{M}_4$: two snapshots with 
observation of $u$]
\label{cor:red4}
In addition to Corollary 2.1, we further assume 
$u, v \in C(\ooo\OOO \times \ooo{I}) \cap C^2(\OOO \times I)$,
$f \ge 0$ in $\OOO$, and  \eqref{eq:5} - \eqref{eq:8}.
If 
$$
u(\cdot,t_0) = \www{u}(\cdot,t_0),\,\, 
v(\cdot,t_0) = \www{v}(\cdot,t_0) \quad \mbox{in $\OOO$} \quad 
\mbox{and} \quad u\vert_{\omega \times I}= \www u\vert_{\omega \times I},
$$
then $f=\widetilde f$ in $\Omega$.
\end{corollary}

So far, assuming the boundary condition \eqref{eq:bc} on the whole 
$\ppp\OOO \times 
(0,T)$, we discuss the Lipschitz stability in the whole spatial 
domain $\OOO$. 
On the other hand, in practice, we do not often know complete 
boundary values on the whole $\ppp\OOO \times (0,T)$, 
but only partial boundary data. In that case, we can prove a H\"older stability estimate under suitable a priori boundedness conditions.

Let $\omega\subset\Omega$ with $\overline{\omega}\subset\Omega$ be fixed, and let
$\Omega_0\subset\Omega$ satisfy
\[
\overline{\omega}\subset\Omega_0\subset\overline{\Omega}_0\subset\Omega.
\]
Fix
\[
0<\delta_0<\min\{t_0,T-t_0\},
\qquad
I_0:=(t_0-\delta_0,\ t_0+\delta_0).
\]
Define
\begin{equation}\label{eq:B-def}
\mathcal{B}:=\sum_{k=0}^2\Bigl(
\norm{\p_t^k(u-\widetilde u)}_{L^2(\omega\times I_0)}^2
+\norm{\p_t^k(v-\widetilde v)}_{L^2(\omega\times I_0)}^2\Bigr)
+\norm{u(\cdot,t_0)-\widetilde u(\cdot,t_0)}_{L^2(\Om)}^2.
\end{equation}
For an arbitrarily chosen constant $\mathcal{N}>0$, we assume 
\begin{equation}\label{eq:N-def}
\sum_{k=0}^1\Bigl(
\norm{\nabla_{x,t}\p_t^k y}_{L^2(\p\Om\times I_0)}^2
+\norm{\nabla_{x,t}\p_t^k z}_{L^2(\p\Om\times I_0)}^2
\end{equation}
$$
+\norm{\nabla_x\p_t^k y(\cdot,t_0\pm\delta_0)}_{L^2(\Om)}^2
+\norm{\nabla_x\p_t^k z(\cdot,t_0\pm\delta_0)}_{L^2(\Om)}^2
\Bigr) \le \mathcal{N}.
$$
We interpret $\mathcal{N}$ as an a priori bound of solutions of $(u,v,f),
(\widetilde u,\widetilde v,\widetilde f)\in\mathcal{U}(M)$.

\begin{theorem}[H\"older stability on interior subdomains]
\label{thm:holder}
Let a subdomain $\Om_0$ satisfy 
$\overline\omega\subset\Om_0\subset\overline{\Om}_0\subset\Om$,
and let $(u,v,f),(\widetilde u,\widetilde v,\widetilde f)\in\mathcal{U}(M)$
solve \eqref{eq:forward}. 
Then, there exist constants
$C=C(M,M_0,t_0,\Om,\Om_0,\omega,T,\delta_0)>0$ and
$\theta=\theta(M,M_0,t_0,\Om,\Om_0,\omega,T,\delta_0)\in(0,1)$ such that
\begin{equation}\label{eq:stab-holder}
\norm{f-\tilde f}_{L^2(\Om_0)}\le C\,\mathcal{B}^{\theta}.
\end{equation}
\end{theorem}

\begin{remark}\label{rem:holder-features}
The constants in the stability estimate depend on $\delta_0$ and deteriorate as $\delta_0\downarrow0$. In particular, the estimate becomes trivial in the limit $\delta_0 \rightarrow 0$.

Several features of Theorem~\ref{thm:holder} are worth noting.
(i)~The stability holds on any compactly contained subdomain
$\Om_0$ satisfying $\ooo{\OOO_0} \subset\Om$; for the H\"older stability,
one cannot take $\Om_0=\Om$ in general, but $\Om_0$
can be chosen arbitrarily close to $\Om$.
(ii)~The data norm $\mathcal{B}$ involves time derivatives up to order two
on $\omega\times I_0$, one order higher than in
Theorem~\ref{thm:reference}. 
However, under a priori boundedness assumptions
of solutions with sufficient orders of Sobolev spaces, we can apply 
the Sobolev interpolation inequality (e.g., Adams \cite{Adams1975Sobolev})
to weaken the data norm for the 
same H\"older stability.
By varying $t_0\in(\varepsilon,T-\varepsilon)$ and repeating the
argument (cf.\ \cite[Theorem~4.5.1]{Y25}), one can also
prove H\"older stability estimate for $u-\widetilde u$ and $v-\widetilde v$ 
in $\Om_0\times (\varepsilon,T-\varepsilon)$.
\end{remark}

Theorem \ref{thm:holder} readily yields the following uniqueness, since $f = \www f$ in any domain $\OOO_0$ such that $\ooo{\OOO_0} \subset \OOO$.
 
\begin{corollary}[Uniqueness from a local interior observation]\label{cor:unique}
Let $(u,v,f)$ and $(\tilde u,\tilde v,\tilde f)$ be two admissible triples.
Assume that
\[
u=\tilde u,\quad v=\tilde v\quad\text{in }\omega\times I_0,
\qquad
u(\cdot,t_0)=\tilde u(\cdot,t_0)\quad\text{in }\Omega.
\]
Then $f=\tilde f$ in $\Omega$. If, in addition, the initial and boundary data coincide, then the forward uniqueness implies $(u,v)=(\tilde u,\tilde v)$ in $Q$.
\end{corollary}


\section{Proof of Theorem \ref{thm:reference}}

\label{sec:proof-lipschitz}

We prove Theorem~\ref{thm:reference} by combining a singular-weight Carleman estimate
for a linearized system with an identification argument at time $t=t_0$. In the proofs, we can choose $I:= (t_0-\delta, \, t_0+\delta)$ such that 
$\delta > 0$ is sufficiently small and $0<\delta<\min\{t_0-t_1,t_2-t_0\}$.
Indeed, it is sufficient that we prove \eqref{eq:stab-reference} with the data norms on $(t_0-\delta, \, t_0+\delta)$.


\medskip\noindent\textbf{Difference system.}

Let $(u,v,f)$ and $(\widetilde u,\widetilde v,\widetilde f)\in\mathcal{U}(M)$ 
solve
\eqref{eq:forward}--\eqref{eq:bc}, where $g, h, a_1, a_2$ are given.
Set
\[
y:=u-\widetilde u,\qquad z:=v-\widetilde v,\qquad F:=f-\widetilde f.
\]
Then $y=z=0$ on $\p\Om\times(0,T)$.
Using the factorisation
\begin{equation}\label{eq:factorisation}
u v^2-\widetilde u\widetilde v^2
=(u-\widetilde u)v^2+\widetilde u(v+\widetilde v)(v-\widetilde v)
=v^2 y + \widetilde u(v+\widetilde v)z,
\end{equation}
we obtain the coupled linear system
\begin{equation}\label{eq:diffsystem}
\begin{cases}
\Dt y = d_1 \Lap y + p_1(x,t)\,y + p_2(x,t)\,z + F(x),\\[3pt]
\Dt z = d_2 \Lap z + p_3(x,t)\,z + p_4(x,t)\,y,
\end{cases}
\qquad (x,t)\in Q,
\end{equation}
with homogeneous Dirichlet boundary conditions
$y=z=0$ on $\p\Om\times(0,T)$,
where the coefficients are
\begin{equation}\label{eq:pcoeff}
\begin{aligned}
p_1&:=-v^2-a_1,\qquad & p_2&:=-\widetilde u(v+\widetilde v),\\
p_3&:=\widetilde u(v+\widetilde v)-a_2,\qquad & p_4&:=v^2.
\end{aligned}
\end{equation}
By the definition of $\mathcal{U}(M,M_0)$, there exists $M_1>0$, depending
only on $M$ and $M_0$, such that
\begin{equation}\label{eq:pk-bounds}
\norm{p_k}_{L^\infty(Q)}+\norm{\p_t p_k}_{L^\infty(Q)}\le M_1,
\qquad k=1,2,3,4.
\end{equation}

\medskip\noindent\textbf{Carleman estimate with singular weight.}

The proof of Theorem~\ref{thm:reference} relies on a Carleman estimate with a singular temporal weight that blows up at the endpoints of the observation interval $I$. We formulate it following
\cite{fursikov1996controllability, Ima1}.

\begin{lemma}[Construction of a weight function]\label{lem:d}
Let a subdomain $\omega_0$ satisfy $\omega_0 \subset \ooo{\omega_0} 
\subset \omega \subset \ooo{\omega} \subset\Omega$.
There exists $d\in C^2(\overline\Omega)$ such that
\[
d>0\ \text{in }\Omega,\qquad d=0\ \text{on }\partial\Omega,\qquad
|\nabla d|>0\ \text{in }\overline{\Omega\setminus\omega_0}.
\]
\end{lemma}

Define the singular temporal factor
\[
\theta(t):=\bigl(t-(t_0-\delta)\bigr)\bigl((t_0+\delta)-t\bigr),
\qquad t\in I.
\]
Note that $\theta>0$ on $I$, the function $\theta$ attains its maximum 
$\delta^2$ at
$t=t_0$, and $\theta(t)\to 0$ as $\theta(t)\to 0$ as $t\to (t_0-\delta)^+$ or $t\to (t_0+\delta)^-$. 

Fix $\lambda>0$ sufficiently large and, for
$(x,t)\in\Omega\times I$, set
\begin{equation}\label{eq:weights-sing}
\varphi(x,t):=\frac{e^{\lambda d(x)}}{\theta(t)},\qquad
\alpha(x,t):=\frac{e^{\lambda d(x)}-e^{2\lambda\|d\|_{C(\overline\Omega)}}}
{\theta(t)}.
\end{equation}
Since $e^{\lambda d(x)}<e^{2\lambda\|d\|_{C(\overline\Om)}}$ for all
$x\in\overline\Om$, the function $\alpha$ is strictly negative and satisfies
\begin{equation}\label{eq:alpha-blowup}
\alpha(x,t)\to -\infty\quad\text{as }\,\,
t\to (t_0-\delta)^+ \text{ or  } t\to (t_0+\delta)^-,
\end{equation}
so that $e^{2s\alpha(x,t)}\to0$ at the endpoints of $I$ for any $s>0$.
Moreover,
\begin{equation}\label{eq:alpha-max}
\alpha(x,t)\le \alpha(x,t_0)\qquad\text{for all }
(x,t)\in\Omega\times I.
\end{equation}

\begin{figure}[htbp]
\centering

\begin{subfigure}[t]{0.36\textwidth}
\centering
\begin{tikzpicture}[>=Stealth,thick,
  x={(0.82cm,-0.28cm)}, y={(0.82cm,0.28cm)}, z={(0cm,1cm)},
  scale=0.72,
  declare function={
  dfunc(\r) = 2.6*(1-\r*\r)*(1-\r*\r)
  + 1.8*exp(-\r*\r*18)*(1-\r*\r);
  }]

 \def\rx{3.0} \def\ry{2.2}

 \draw[gray!45,thin,dashed]
  plot[domain=0:180,samples=60]
  ({\rx*cos(\x)},{\ry*sin(\x)},0);

 \foreach \ang in {15,35,...,165}{
  \draw[gray!25,very thin]
  plot[domain=0:0.97,samples=30,smooth,variable=\q]
 ({\rx*\q*cos(\ang)},{\ry*\q*sin(\ang)},{dfunc(\q)});
 }
 \foreach \ang in {195,215,...,345}{
  \draw[gray!45,thin]
  plot[domain=0:0.97,samples=30,smooth,variable=\q]
 ({\rx*\q*cos(\ang)},{\ry*\q*sin(\ang)},{dfunc(\q)});
 }

 \draw[very thick, red!70!black]
  plot[domain=0:0.98,samples=80,smooth,variable=\q]
  ({-\rx*\q},0,{dfunc(\q)});
 \draw[very thick, red!70!black]
  plot[domain=0:0.98,samples=80,smooth,variable=\q]
  ({\rx*\q},0,{dfunc(\q)});

 \draw[gray!35,thin,densely dashed]
  plot[domain=180:360,samples=50]
  ({0.5*\rx*cos(\x)},{0.5*\ry*sin(\x)},{dfunc(0.5)});

 \draw[very thick]
  plot[domain=180:360,samples=60]
  ({\rx*cos(\x)},{\ry*sin(\x)},0);

 \draw[blue!70!black, semithick, fill=blue!10, fill opacity=0.4]
  plot[domain=0:360,samples=60]
  ({0.55*cos(\x)},{0.42*sin(\x)},0) -- cycle;
 \node[blue!70!black,font=\scriptsize] at (0.65,0.5,-0.25) {$\omega_0$};

 \draw[->,thin,red!70!black] (0.4,1.2,5.2)
  node[above,font=\small]{$d(x)$}
  -- (0.05,0.05,{dfunc(0)});
 \node[font=\scriptsize] at (0,-\ry-0.6,-0.2) {$d=0$ on $\partial\Omega$};
 \node[font=\scriptsize] at (-\rx-0.5,0,1.8) {$d>0$};
 \node[font=\scriptsize] at (\rx+0.4,-0.2,0) {$\Omega$};
 \draw[<-,thin,gray]
  ({0.7*\rx*cos(250)},{0.7*\ry*sin(250)},{dfunc(0.7)+0.15})
  -- (-1.8,-1,3.5)
  node[above,black,font=\scriptsize,align=center]
  {$|\nabla d|>0$ in $\overline\Omega\!\setminus\!\omega_0$};

\end{tikzpicture}
\caption{Three-dimensional view of the weight function $d(x)$.}

\label{fig:dx-3d}
\end{subfigure}
\hfill
\begin{subfigure}[t]{0.62\textwidth}
\centering
\begin{tikzpicture}[>=Stealth,thick,
 declare function={
  theta(\t) = \t*(4-\t) ;
  phifun(\t) = min(3.5, 5.0/theta(\t) ) ;
  efun(\t) = 2.2*exp( -2.5/max(0.06,theta(\t)) + 2.5/4 ) ;
 },
 scale=0.65]

 \begin{scope}[xshift=0cm]
  \draw[->] (-0.5,0) -- (5.3,0) node[right,font=\scriptsize]{$t$};
  \draw[->] (0,-0.3) -- (0,3.6);

  \draw[very thick, blue!70!black]
  plot[domain=0.02:3.98,samples=100,smooth]
 (\x,{ 0.65*theta(\x) });

  \foreach \pos/\lab in {0/{t_0\!-\!\delta}, 2/{t_0}, 4/{t_0\!+\!\delta}} {
  \draw (\pos,0.08)--(\pos,-0.08) node[below,font=\scriptsize]{$\lab$};
  }

  \draw[densely dashed,gray,thin] (0,{0.65*4}) -- (2,{0.65*4});
  \draw[densely dashed,gray,thin] (2,0) -- (2,{0.65*4});
  \fill (2,{0.65*4}) circle(1.5pt);
  \node[left,font=\scriptsize] at (-0.05,{0.65*4}) {$\delta^2$};

  \fill (0,0) circle(2pt);
  \fill (4,0) circle(2pt);

  \node[blue!70!black,font=\scriptsize,align=center]
  at (2,3.75)
  {$\theta(t)=(t\!-\!t_0\!+\!\delta)(t_0\!+\!\delta\!-\!t)$};

  \node[font=\footnotesize,gray] at (2,-1.0) {(i)};
 \end{scope}

 \begin{scope}[xshift=7.5cm]
  \draw[->] (-0.5,0) -- (5.3,0) node[right,font=\scriptsize]{$t$};
  \draw[->] (0,-0.3) -- (0,4.2);

  \draw[very thick, red!70!black]
  plot[domain=0.38:3.62,samples=120,smooth]
 (\x,{ phifun(\x) });

  \draw[->,red!70!black,semithick] (0.40,3.4) -- (0.18,3.9);
  \draw[->,red!70!black,semithick] (3.60,3.4) -- (3.82,3.9);

  \draw[red!30,thin,densely dotted] (0,0) -- (0,4.0);
  \draw[red!30,thin,densely dotted] (4,0) -- (4,4.0);

  \draw[very thick, green!50!black, densely dashed]
  plot[domain=0.05:3.95,samples=120,smooth]
 (\x,{ efun(\x) });

  \foreach \pos/\lab in {0/{t_0\!-\!\delta}, 2/{t_0}, 4/{t_0\!+\!\delta}} {
  \draw (\pos,0.08)--(\pos,-0.08) node[below,font=\scriptsize]{$\lab$};
  }

  \draw[rounded corners=2pt, fill=white, draw=gray!50, thin]
  (2.55,3.15) rectangle (5.75,4.5);
  \draw[very thick, red!70!black] (2.75,4.1) -- (3.35,4.1);
  \node[right,font=\scriptsize] at (3.4,4.1)
  {$\varphi=\frac{e^{\lambda d}}{\theta}$};
  \draw[very thick, green!50!black, densely dashed]
  (2.75,3.5) -- (3.35,3.5);
  \node[right,font=\scriptsize] at (3.4,3.5) {$e^{2s\alpha}\!\to\!0$};

  \node[font=\footnotesize,gray] at (2,-1.0) {(ii)};
 \end{scope}

\end{tikzpicture}
\caption{Temporal weight components on $I=(t_0-\delta,t_0+\delta)$.
(i)~The cut-off function $\theta(t)=(t-t_0+\delta)(t_0+\delta-t)$
vanishes at the endpoints $t=t_0-\delta$ and $t=t_0+\delta$.
(ii)~The weight satisfies $\varphi=e^{\lambda d}/\theta\to\infty$ and $e^{2s\alpha}\to 0$ as
$t\to (t_0-\delta)^+$ or $t\to (t_0+\delta)^-$.}
\label{fig:singular-weight}
\end{subfigure}

\caption{Carleman weight functions for the singular-weight estimate.}
\label{fig:weights}
\end{figure}
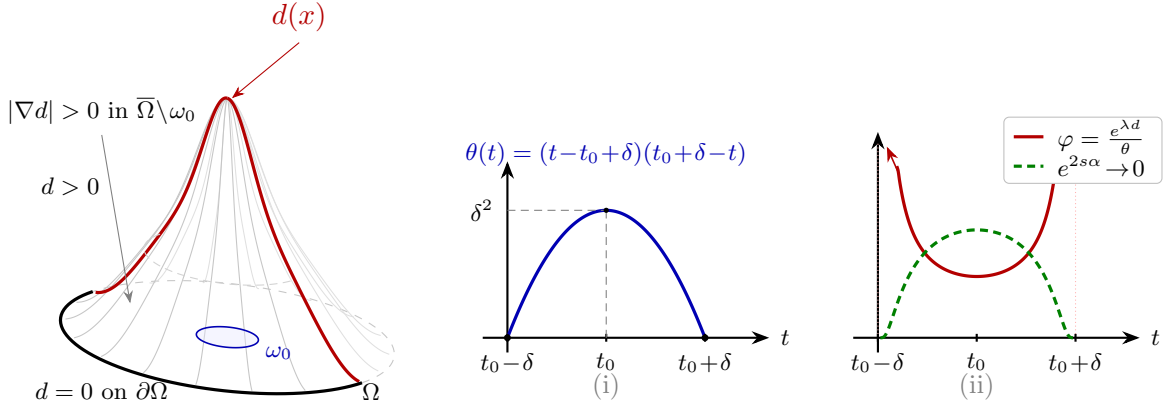

\begin{lemma}[Estimate on $\p_t\alpha$]\label{lem:dtalpha}
There exists $C>0$, depending on $\delta$, $\lambda$, and
$\norm{d}_{C(\overline\Om)}$, such that
\begin{equation}\label{eq:dtalpha}
|\partial_t\alpha(x,t)|\le C\,\varphi(x,t)^2
\qquad\text{for all }(x,t)\in\Omega\times I.
\end{equation}
\end{lemma}

\begin{proof}
Write 
$$
A(x):=e^{\lambda d(x)}-e^{2\lambda\|d\|_{C(\overline\Om)}}.
$$
Then 
$\alpha=A/\theta$ and $\p_t\alpha = -A\,\theta'/\theta^2.$
Since $$|\theta'(t)|=|2t_0-2t|\le 2\delta \text{ for } t\in I \text{, and }
|A(x)|\le e^{2\lambda\|d\|_{C(\overline\Om)}},$$ 
while
$\varphi^2=e^{2\lambda d}/\theta^2\ge 1/\theta^2$, we obtain
$$|\p_t\alpha|\le 2\delta\,
e^{2\lambda\|d\|_{C(\overline\Om)}}\,\varphi^2.$$
\end{proof}

Define the Carleman functional
\[
\mathcal I_s(w):=\int_{\Omega\times I}\left(\frac{1}{s\varphi}|\partial_t w|^2
+\frac{1}{s\varphi}\sum_{i,j=1}^n|\p_i\p_j w|^2
+s\varphi|\nabla w|^2+s^3\varphi^3|w|^2\right)e^{2s\alpha}\,dx\,dt.
\]

\begin{lemma}[Carleman estimate with singular weight]\label{lem:carleman}
Let $d$ and $\alpha,\varphi$ be as above.
There exist constants $\lambda_0>0$, $s_0>0$, and $C>0$ such that for all
$\lambda\ge\lambda_0$ and $s\ge s_0$, the following holds for any
$w\in H^{2,1}(\Omega\times I)$ with $w=0$ on $\partial\Omega\times I$:
\begin{equation}\label{eq:carleman-scalar}
\mathcal I_s(w)\le
C\int_{\Omega\times I}
\bigl|\partial_t w-d_k\Delta w + B w\bigr|^2 e^{2s\alpha}\,dx\,dt
+C\int_{\omega_0\times I} s^3\varphi^3 |w|^2 e^{2s\alpha}\,dx\,dt,
\end{equation}
where $d_k\in\{d_1,d_2\}$ and $B\in L^\infty(\Om\times I)$.
\end{lemma}

This is a standard Carleman inequality for parabolic operators with interior observation. The singular behavior of $\theta$ at the endpoints of $I$ ensures that the time-boundary contributions vanish after the integration-by-parts argument.
The proof can be found in \cite{fursikov1996controllability, Ima1}.

\begin{proof}[Proof of Theorem~\ref{thm:reference}]
We divide the argument into four steps. All constants $C>0$ depend on $\Omega,\omega,T,t_0,\delta,M,M_0$ and may change from line to line.

\medskip
\noindent\textbf{Step 1: Carleman estimate for the coupled system.}

Set $y_1:=\p_t y$ and $z_1:=\p_t z$.
Differentiating \eqref{eq:diffsystem} with respect to $t$, yields
\begin{equation}\label{eq:diffsystem-dt}
\begin{cases}
\p_t y_1 = d_1\Lap y_1 + p_1 y_1 + p_2 z_1
 + (\p_t p_1) y + (\p_t p_2) z,\\[3pt]
\p_t z_1 = d_2\Lap z_1 + p_3 z_1 + p_4 y_1
 + (\p_t p_3) z + (\p_t p_4) y,
\end{cases}
\qquad (x,t)\in \Omega\times I,
\end{equation}
with $y_1=z_1=0$ on $\p\Om\times I$.

We apply Lemma~\ref{lem:carleman} to each of the four functions $y,z,y_1,z_1$:
for $y$ and $y_1$ we use diffusion constant $d_1$, and for $z$ and $z_1$ we
use $d_2$.
The right-hand sides of the Carleman estimates involve:
for $y$, $$|\p_t y-d_1\Lap y+(-p_1)y|^2=|p_2 z+F|^2;$$
for $z$, $$|\p_t z-d_2\Lap z+(-p_3)z|^2=|p_4 y|^2;$$
and similarly for $y_1$ and $z_1$ using \eqref{eq:diffsystem-dt}.

Summing the four inequalities and using \eqref{eq:pk-bounds} together with
the elementary estimate $$|A+B|^2\le2|A|^2+2|B|^2,$$ we obtain, for all
sufficiently large $s$,
\begin{align}
&\int_{\Omega\times I}\Biggl[
\frac{1}{s\varphi}\Bigl(|\p_t^2 y|^2+|\p_t^2 z|^2\Bigr)
+s^3\varphi^3\sum_{\ell=0}^1
\Bigl(|\p_t^\ell y|^2+|\p_t^\ell z|^2\Bigr)
\Biggr]e^{2s\alpha}\,dx\,dt \notag\\
&\qquad\le
C\int_{\Omega\times I}|F|^2 e^{2s\alpha}\,dx\,dt
+C\int_{\omega\times I}s^3\varphi^3\sum_{\ell=0}^1
\Bigl(|\p_t^\ell y|^2+|\p_t^\ell z|^2\Bigr)e^{2s\alpha}\,dx\,dt
\notag\\
&\qquad\quad
+C\int_{\Omega\times I}
\Bigl(|y|^2+|z|^2\Bigr)e^{2s\alpha}\,dx\,dt.
\label{eq:carleman-preabsorb}
\end{align}
The last integral arises from the lower-order terms
$(\p_t p_k)y$ and $(\p_t p_k)z$ in \eqref{eq:diffsystem-dt}.
Since the left-hand side contains
$s^3\varphi^3(|y|^2+|z|^2)e^{2s\alpha}$, choosing $s$ sufficiently large
absorbs the last integral on the right-hand side, yielding
\begin{align}
&\int_{\Omega\times I}\Biggl[
\frac{1}{s\varphi}\Bigl(|\p_t^2 y|^2+|\p_t^2 z|^2\Bigr)
+s^3\varphi^3\sum_{\ell=0}^1
\Bigl(|\p_t^\ell y|^2+|\p_t^\ell z|^2\Bigr)
\Biggr]e^{2s\alpha}\,dx\,dt \notag\\
&\qquad\le
C\int_{\Omega\times I}|F|^2 e^{2s\alpha}\,dx\,dt
+C\int_{\omega\times I}s^3\varphi^3\sum_{\ell=0}^1
\Bigl(|\p_t^\ell y|^2+|\p_t^\ell z|^2\Bigr)e^{2s\alpha}\,dx\,dt.
\label{eq:carleman-system}
\end{align}
Define the observation quantity
\begin{equation}\label{eq:D1-def}
D_1^2:=\int_{\omega\times I}\sum_{\ell=0}^1
\Bigl(|\p_t^\ell y|^2+|\p_t^\ell z|^2\Bigr)\,dx\,dt.
\end{equation}
Since $$\sup_{(x,t)\in\overline\Om\times I}
(s^3\varphi^3 e^{2s\alpha})\le e^{Cs}$$ for all $s\ge1$, we have
\begin{equation}\label{eq:obs-bound}
\int_{\omega\times I}s^3\varphi^3\sum_{\ell=0}^1
\Bigl(|\p_t^\ell y|^2+|\p_t^\ell z|^2\Bigr)e^{2s\alpha}\,dx\,dt
\le C e^{Cs} D_1^2.
\end{equation}

\medskip
\noindent\textbf{Step 2: Weighted estimate of $\p_t y(\cdot,t_0)$.}

Since $e^{2s\alpha(x,t)}\to0 $ as $ t\to t_0$, we write
\begin{align}
\int_\Om |\p_t y(x,t_0)|^2 e^{2s\alpha(x,t_0)}\,dx
&=\int_{t_0-\delta}^{t_0}\p_t\left(
\int_\Om |\p_t y|^2 e^{2s\alpha}\,dx\right)dt \notag\\
&=\int_{t_0-\delta}^{t_0}\int_\Om
\Bigl(2(\p_t y)(\p_t^2 y)+2s(\p_t\alpha)|\p_t y|^2\Bigr)
e^{2s\alpha}\,dx\,dt.
\label{eq:dt-y-identity}
\end{align}
By the Cauchy--Schwarz inequality with the factorisation
$$(\p_t y)(\p_t^2 y)=
(s\sqrt\varphi\,\p_t y)\cdot(s^{-1}\varphi^{-1/2}\,\p_t^2 y)$$
and the inequality $2ab\le a^2+b^2$, we obtain
\begin{equation}\label{eq:young-dty}
2|(\p_t y)(\p_t^2 y)|
\le s^2\varphi|\p_t y|^2+\frac{1}{s^2\varphi}|\p_t^2 y|^2.
\end{equation}
For the first term, using $\varphi\ge\delta^{-2}$, we have
$$s^2\varphi = (s^2\varphi)/(s^3\varphi^3)\cdot s^3\varphi^3
= (s\varphi^2)^{-1}\cdot s^3\varphi^3
\le C s^{-1}\cdot s^3\varphi^3.$$
For the second,
$s^{-2}\varphi^{-1}=s^{-1}\cdot(s\varphi)^{-1}$.
Thus,
\begin{equation}\label{eq:young-dty-final}
2|(\p_t y)(\p_t^2 y)|
\le \frac{C}{s}\left(s^3\varphi^3|\p_t y|^2
+\frac{1}{s\varphi}|\p_t^2 y|^2\right).
\end{equation}
For the second term in \eqref{eq:dt-y-identity}, using
Lemma~\ref{lem:dtalpha} and
$s\varphi^2\le Cs^{-1}\cdot s^3\varphi^3$ for large $s$, we obtain
$$2s|\p_t\alpha|\,|\p_t y|^2\le Cs^{-1}\cdot s^3\varphi^3|\p_t y|^2.$$

Combining with \eqref{eq:carleman-system} and \eqref{eq:obs-bound}, we
conclude
\begin{equation}\label{eq:dt-y-final}
\int_\Om |\p_t y(x,t_0)|^2 e^{2s\alpha(x,t_0)}\,dx
\le \frac{C}{s}\int_{\Om\times I}|F|^2 e^{2s\alpha}\,dx\,dt
+ C e^{Cs} D_1^2
\end{equation}
for all $s\ge s_0$.

\medskip
\noindent\textbf{Step 3: Weighted estimate of $z(\cdot,t_0)$.}

Similarly, since $e^{2s\alpha(x,t)}\to0$ as $t\to (t_0-\delta)^+$, we have
\begin{align}
\int_\Om |z(x,t_0)|^2 e^{2s\alpha(x,t_0)}\,dx
&=\int_{t_0-\delta}^{t_0}\int_\Om
\Bigl(2z\,\p_t z+2s(\p_t\alpha)|z|^2\Bigr)
e^{2s\alpha}\,dx\,dt.
\label{eq:z-identity}
\end{align}
For the first term, using the elementary inequality $2ab\le a^2+b^2$:
\[
2|z\,\p_t z|\le |z|^2+|\p_t z|^2.
\]
For the second term, Lemma~\ref{lem:dtalpha} gives
$2s|\p_t\alpha|\,|z|^2\le Cs\varphi^2|z|^2$.
Hence,
\[
2|z\,\p_t z|+2s|\p_t\alpha|\,|z|^2
\le (1+Cs\varphi^2)|z|^2+|\p_t z|^2.
\]
For $s\varphi\ge 1$ (which holds for $s$ sufficiently large, since
$\varphi\ge\delta^{-2}$), we have
\[
1+Cs\varphi^2\le \frac{C}{s}\,s^3\varphi^3,\qquad
1\le \frac{C}{s}\,s^3\varphi^3.
\]
Therefore,
\begin{equation}\label{eq:z-bound}
\int_\Om |z(x,t_0)|^2 e^{2s\alpha(x,t_0)}\,dx
\le \frac{C}{s}\int_{\Om\times I} s^3\varphi^3
\left(|z|^2+|\p_t z|^2\right)e^{2s\alpha}\,dx\,dt.
\end{equation}
The right-hand side is controlled by \eqref{eq:carleman-system}, so that
\begin{equation}\label{eq:z-final}
\int_\Om |z(x,t_0)|^2 e^{2s\alpha(x,t_0)}\,dx
\le \frac{C}{s}\int_{\Om\times I}|F|^2 e^{2s\alpha}\,dx\,dt
+ C e^{Cs} D_1^2.
\end{equation}
\medskip
\noindent\textbf{Step 4: Source identity at $t_0$ and conclusion.}

Evaluating the first equation of \eqref{eq:diffsystem} at $t=t_0$ gives
\begin{equation}\label{eq:source-identity}
F(x)=\p_t y(x,t_0)-d_1\Lap y(x,t_0)
-p_1(x,t_0)y(x,t_0)-p_2(x,t_0)z(x,t_0),
\end{equation}
whence
\begin{equation}\label{eq:F-pointwise}
|F(x)|^2\le C\Bigl(|\p_t y(x,t_0)|^2+|\Lap y(x,t_0)|^2
+|y(x,t_0)|^2+|z(x,t_0)|^2\Bigr).
\end{equation}
Define
\begin{equation}\label{eq:D2-def}
D_2:=\norm{y(\cdot,t_0)}_{H^2(\Om)}
=\norm{u(\cdot,t_0)-\tilde u(\cdot,t_0)}_{H^2(\Om)}.
\end{equation}
Multiplying \eqref{eq:F-pointwise} by $e^{2s\alpha(x,t_0)}$, integrating
over $\Om$, and using $\alpha(\cdot,t_0)$ bounded above on $\overline\Om$
for the terms involving $y(\cdot,t_0)$, we obtain
\begin{equation}\label{eq:F-weighted1}
\int_\Om |F(x)|^2 e^{2s\alpha(x,t_0)}\,dx
\le C e^{Cs} D_2^2
+C\int_\Om\Bigl(|\p_t y(x,t_0)|^2+|z(x,t_0)|^2\Bigr)
e^{2s\alpha(x,t_0)}\,dx.
\end{equation}
Substituting \eqref{eq:dt-y-final} and \eqref{eq:z-final} into
\eqref{eq:F-weighted1} yields
\begin{equation}\label{eq:F-weighted2}
\int_\Om |F(x)|^2 e^{2s\alpha(x,t_0)}\,dx
\le C e^{Cs}(D_1^2+D_2^2)
+\frac{C}{s}\int_{\Om\times I}|F|^2 e^{2s\alpha}\,dx\,dt.
\end{equation}
Since $e^{2s\alpha(x,t)}\le e^{2s\alpha(x,t_0)}$ by \eqref{eq:alpha-max}
and $F=F(x)$ is independent of $t$,
\[
\int_{\Om\times I}|F(x)|^2 e^{2s\alpha(x,t)}\,dx\,dt
\le |I|\int_\Om |F(x)|^2 e^{2s\alpha(x,t_0)}\,dx.
\]
Choosing $s\ge s_0$ so large that $C|I|/s\le 1/2$ absorbs the last term into
the left-hand side, yielding
\[
\int_\Om |F(x)|^2 e^{2s\alpha(x,t_0)}\,dx
\le C e^{Cs}(D_1^2+D_2^2).
\]
Since $\alpha(\cdot,t_0)$ is bounded below on $\overline\Om$, there exists
$c_*>0$ such that $e^{2s\alpha(x,t_0)}\ge e^{-c_* s}$ for all
$x\in\Om$.
Hence
$\norm{F}_{L^2(\Om)}^2\le C e^{Cs}(D_1^2+D_2^2)$,
and therefore
$$\norm{f-\tilde f}_{L^2(\Om)}\le C(D_1+D_2),$$
which is \eqref{eq:stab-reference}.
\end{proof}


\section{Proof of Theorem \ref{thm:holder}}
\label{sec:proof-holder}


In this section, we establish a H\"older-type stability estimate using a
Carleman estimate with a regular (non-singular) weight function.
Following \cite[Chapter~4]{Y25}, we will prove 
stability on any interior subdomain $\Om_0$ of $\Om$ at the cost of a
H\"older exponent strictly less than one.

Let $\Om_0$ be an arbitrary subdomain satisfying
$\overline\omega\subset\Om_0\subset\overline{\Om}_0\subset\Om$.
Since $\overline{\Om}_0\subset\Om$, the function $d$ from Lemma~\ref{lem:d}
satisfies
\[
\min_{x\in\overline{\Om}_0}d(x)>0.
\]

Fix $\delta_0\in(0,\min\{t_0,T-t_0\})$ and set
$I_0:=(t_0-\delta_0,\;t_0+\delta_0)$,
$Q_{I_0}:=\Om\times I_0$.
Choose $\beta>0$ sufficiently large and $r\in(0,\delta_0)$
sufficiently small such that
\begin{equation}\label{eq:beta-r}
r^2<\frac{1}{\beta}\min_{x\in\overline{\Om}_0}d(x),\qquad
\beta>\frac{1}{\delta_0^2-r^2}
\left(\max_{x\in\overline\Om}d(x)
-\min_{x\in\overline{\Om}_0}d(x)\right).
\end{equation}
Such a choice is possible: first take $r>0$ small enough that
\[
\frac{r^2}{\delta_0^2-r^2}
<\frac{\min_{x\in\overline{\Om}_0}d(x)}
{\max_{x\in\overline\Om}d(x)-\min_{x\in\overline{\Om}_0}d(x)},
\]
and then choose $\beta>0$ to satisfy both inequalities in \eqref{eq:beta-r}
simultaneously.

Define a regular weight function
\begin{equation}\label{eq:weights-reg}
\psi(x,t):=e^{\lambda(d(x)-\beta(t-t_0)^2)},
\qquad (x,t)\in Q_{I_0},
\end{equation}
where $\lambda>0$ is taken sufficiently large.
Set
\begin{equation}\label{eq:rho-def}
\rho_1:=\max\bigl\{1,\;
\exp\bigl(\lambda(\max_{x\in\overline\Om}d(x)
-\beta\delta_0^2)\bigr)\bigr\},\qquad
\rho_2:=\min_{x\in\overline{\Om}_0,\;|t-t_0|\le r}\psi(x,t).
\end{equation}
By the choice \eqref{eq:beta-r}, the following separation property holds
\begin{equation}\label{eq:rho-separation}
\rho_2>\rho_1.
\end{equation}

\begin{figure}[htbp]
\centering
\begin{tikzpicture}[>=Stealth,thick,
 declare function={
  psiA(\t) = exp(0.65*(2.0 - 0.22*(\t-3)*(\t-3)));
  psiB(\t) = exp(0.65*(1.2 - 0.22*(\t-3)*(\t-3)));
  psiC(\t) = exp(0.65*(0.1 - 0.22*(\t-3)*(\t-3)));
 }]

 \draw[->] (-0.5,0) -- (7.2,0) node[right,font=\small]{$t$};
 \draw[->] (0,-0.3) -- (0,5.5) node[above,font=\small]{$\psi$};

 \draw[thick, blue!55, densely dashed]
  plot[domain=0.3:5.7,samples=100,smooth] (\x,{psiC(\x)});
 \draw[thick, gray!65]
  plot[domain=0.3:5.7,samples=100,smooth] (\x,{psiB(\x)});
 \draw[very thick, red!70!black]
  plot[domain=0.3:5.7,samples=100,smooth] (\x,{psiA(\x)});

 \draw[thin] (0.5,0.08)--(0.5,-0.08)
  node[below,font=\small]{$t_0\!-\!\delta_0$};
 \draw[thin] (5.5,0.08)--(5.5,-0.08)
  node[below,font=\small]{$t_0\!+\!\delta_0$};
 \draw[thin] (3,0.08)--(3,-0.08)
  node[below,font=\small]{$t_0$};

 \draw[thick,orange!75!black,<->] (2.1,-0.65) -- (3.9,-0.65);
 \node[below,orange!75!black,font=\small] at (3,-0.65)
  {$|t\!-\!t_0|\le r$};

 \pgfmathsetmacro{\rhoone}{psiC(0.5)}
 \draw[densely dashed,gray!70] (0,\rhoone) -- (6.7,\rhoone)
  node[right,font=\small,black]{$\rho_1$};
 \pgfmathsetmacro{\rhotwo}{psiA(3.9)}
 \draw[densely dashed,gray!70] (0,\rhotwo) -- (6.7,\rhotwo)
  node[right,font=\small,black]{$\rho_2$};

 \draw[<->,very thick,black!80]
  (6.4,\rhoone) -- (6.4,\rhotwo);
 \node[right,font=\small] at (6.5,{0.5*\rhoone+0.5*\rhotwo})
  {$\rho_2\!>\!\rho_1$};

 \fill[red!8,opacity=0.4]
  (2.1,\rhotwo) rectangle (3.9,{psiA(3)+0.15});

 \draw[rounded corners=2pt, fill=white, draw=gray!50, thin]
  (4.3,4.0) rectangle (7.0,5.3);
 \draw[very thick, red!70!black] (4.5,5.0) -- (5.2,5.0);
 \node[right,font=\small] at (5.25,5.0)
  {$x\!\in\!\Omega_0$};
 \draw[thick, gray!65] (4.5,4.6) -- (5.2,4.6);
 \node[right,font=\small] at (5.25,4.6)
  {$x\!\in\!\Omega$};
 \draw[thick, blue!55, densely dashed] (4.5,4.2) -- (5.2,4.2);
 \node[right,font=\small] at (5.25,4.2)
  {$x\!\in\!\partial\Omega$};

\end{tikzpicture}
\caption{The regular weight
$\psi(x,t)=e^{\lambda(d(x)-\beta(t-t_0)^2)}$
at three representative locations.}
\label{fig:regular-weight}
\end{figure}
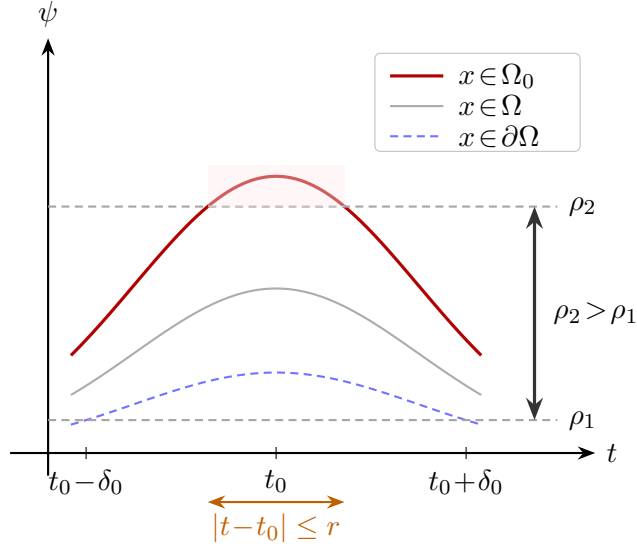

\begin{lemma}[Carleman estimate with regular weight]\label{lem:carleman-reg}
Let $\psi$ be defined by \eqref{eq:weights-reg}.
There exist constants $\lambda_0>0$, $s_0>0$, and $C>0$ such that for all
$\lambda\ge\lambda_0$ and $s\ge s_0$, the following holds for any
$w\in H^{2,1}(Q_{I_0})$ with $w=0$ on $\p\Om\times I_0$:
\begin{align}\label{eq:carleman-reg}
&\int_{Q_{I_0}}\Bigl(\frac{1}{s}|\p_t w|^2
+\frac{1}{s}\sum_{i,j=1}^n|\p_i\p_j w|^2
+s|\nabla w|^2+s^3|w|^2\Bigr)e^{2s\psi}\,dx\,dt\notag\\
&\qquad\le C\int_{Q_{I_0}}|\mathcal{L}w|^2 e^{2s\psi}\,dx\,dt
+C\int_{\omega_0\times I_0}s^3|w|^2 e^{2s\psi}\,dx\,dt\notag\\
&\qquad\quad
+C\Bigl(\norm{\nabla_{x,t}w}_{L^2(\p\Om\times I_0)}^2
+\norm{\nabla_x w(\cdot,t_0\pm\delta_0)}_{L^2(\Om)}^2\Bigr)
e^{2s\rho_1},
\end{align}
where $\mathcal{L}w:=\p_t w-d_k\Lap w+B(x,t)w$ with
$d_k\in\{d_1,d_2\}$ and $B\in L^\infty(Q_{I_0})$.
The last line accounts for the boundary and temporal-endpoint
contributions, which do not vanish since the weight $\psi$ remains
bounded on $\overline{Q}_{I_0}$.
\end{lemma}

The main difference from Lemma~\ref{lem:carleman} is that the weight $\psi$
does not blow up at the endpoints of $I_0$, so that the time-boundary
contributions do not automatically vanish and must be estimated using the
a priori bounds.
The proof of the lemma is found in \cite[Theorem~4.3.2]
{Y25}, for example.

\begin{proof}[Proof of Theorem~\ref{thm:holder}]
Set $y:=u-\widetilde u$, $z:=v-\widetilde v$, $F:=f-\widetilde f$,
$y_0(x):=y(x,t_0)$, and $y_1:=\p_t y$, $z_1:=\p_t z$.

\medskip
\noindent\textbf{Step 1: Representation using the integral relation.}

Since $y(x,t)=\int_{t_0}^t y_1(x,\eta)\,d\eta+y_0(x)$, we can rewrite
the time-differentiated equations.
Differentiating \eqref{eq:diffsystem} in $t$, we obtain
\begin{equation}\label{eq:y1-eq}
\p_t y_1=d_1\Lap y_1+p_1 y_1
+\biggl\{p_2 z_1+(\p_t p_1)\int_{t_0}^t y_1(x,\eta)\,d\eta
+(\p_t p_2)z+(\p_t p_1)y_0\biggr\},
\end{equation}
\begin{equation}\label{eq:z1-eq}
\p_t z_1=d_2\Lap z_1+p_3 z_1
+\biggl\{p_4 y_1+(\p_t p_3)z
+(\p_t p_4)\int_{t_0}^t y_1(x,\eta)\,d\eta+(\p_t p_4)y_0\biggr\},
\end{equation}
and the equation for $z$ reads
\begin{equation}\label{eq:z-eq-holder}
\p_t z=d_2\Lap z+p_3 z+p_4 y\qquad\text{in }Q_{I_0}.
\end{equation}
A key feature is that we apply the Carleman estimate to the system
\eqref{eq:y1-eq}--\eqref{eq:z-eq-holder} in the three unknowns
$(y_1,z_1,z)$, \emph{without} including $y$ itself.
This is possible because the terms involving $y$ appear either through
$y_0(x)$ (a known datum) or through the integral
$\int_{t_0}^t y_1(x,\eta)\,d\eta$, which is controlled by $y_1$.

\medskip
\noindent\textbf{Step 2: Integral estimate.}

The following estimate is needed to handle the integral terms.

\begin{lemma}\label{lem:integral}
For all $s>0$ and $w\in L^2(\Om\times I_0)$,
\[
\int_{Q_{I_0}}\left|\int_{t_0}^t w(x,\eta)\,d\eta\right|^2
e^{2s\psi(x,t)}\,dx\,dt
\le C\int_{Q_{I_0}}|w(x,\eta)|^2 e^{2s\psi(x,\eta)}\,d\eta\,dx.
\]
\end{lemma}

\begin{proof}
First we have $\psi(x, t) \leq \psi(x, \eta)$ for $t_0 \leq \eta \leq t \leq t_0+\delta_0$ or $t_0-\delta_0 \leq t \leq \eta \leq t_0$. Let $t_0 \leq t \leq t_0+\delta_0$. Then, using the Cauchy--Schwarz inequality, we estimate
$$
\begin{aligned}
& \int_{t_0}^{t_0+\delta_0}\left|\int_{t_0}^t w(x, \eta)\, d \eta\right|^2 e^{2 s \psi(x, t)}\, d t \leq \int_{t_0}^{t_0+\delta_0}\left(\int_{t_0}^t|w(x, \eta)|^2\, d \eta\right)\left|t-t_0\right| e^{2 s \psi(x, t)}\, d t \\
\leq{}& \delta_0 \int_{t_0}^{t_0+\delta_0}\left(\int_{t_0}^t|w(x, \eta)|^2 e^{2 s \psi(x, \eta)}\, d \eta\right) d t \leq \delta_0^2 \int_{t_0}^{t_0+\delta_0}|w(x, \eta)|^2 e^{2 s \psi(x, \eta)}\, d \eta.
\end{aligned}
$$
The estimation for $t_0-\delta_0 \leq t \leq t_0$ is the same. Thus the proof is complete.
\end{proof}

\medskip
\noindent\textbf{Step 3: Carleman estimate and absorption.}

Applying Lemma~\ref{lem:carleman-reg} to each of the three functions
$y_1$, $z_1$, and $z$ in the system
\eqref{eq:y1-eq}--\eqref{eq:z-eq-holder},
using Lemma~\ref{lem:integral} to control the integral terms
$\int_{t_0}^t y_1(x,\eta)\,d\eta$,
and summing the resulting inequalities, we obtain:
there exist $s_0>0$ and $C>0$ such that for all $s\ge s_0$,

\begin{align}
&\int_{Q_{I_0}}\Biggl\{\frac{1}{s}\Biggl(
\sum_{k=1}^2\bigl(|\p_t^k y|^2+|\p_t^k z|^2\bigr)
+\sum_{k=0}^1\sum_{i,j=1}^n
\bigl(|\p_t^k\p_i\p_j y|^2+|\p_t^k\p_i\p_j z|^2
\bigr)\Biggr)\notag\\
&\qquad\qquad\qquad
+s\sum_{k=0}^1\bigl(|\nabla\p_t^k y|^2+|\nabla\p_t^k z|^2\bigr)
+s^3\sum_{k=0}^1\bigl(|\p_t^k y|^2+|\p_t^k z|^2\bigr)
\Biggr\}e^{2s\psi}\,dx\,dt\notag\\
&\quad\le
Ce^{Cs}\,\mathcal{B}
+Ce^{2s\rho_1}\,\mathcal{N}.
\label{eq:holder-carleman-full}
\end{align}

Note that the source $F(x)$ does not appear on the right-hand side:
the time differentiation eliminates $F$ from the equations for $y_1$ and $z_1$,
because $F$ does not appear in the equation for $z$. The terms involving $y$ on the left-hand side arise via the integral relation $y=\int_{t_0}^t y_1\,d\eta+y_0$ and Lemma~\ref{lem:integral}. The terms $(\p_t p_k)y$ and $(\p_t p_k)z$ on the right-hand sides of \eqref{eq:y1-eq}--\eqref{eq:z1-eq} produce lower-order 
contribution that are absorbed into the left-hand side by choosing $s$ sufficiently large; see the analogous absorption in \eqref{eq:carleman-preabsorb}.

Using $\psi(x,t)\ge\rho_2$ on
$\overline{\Om}_0\times[t_0-r,t_0+r]$, by the separation
\eqref{eq:rho-separation}, we deduce from \eqref{eq:holder-carleman-full}

\begin{align}
&e^{2s\rho_2}\int_{\Om_0\times(t_0-r,t_0+r)}
\Biggl\{\frac{1}{s}\Biggl(
\sum_{k=1}^2\Bigl(|\p_t^k y|^2+|\p_t^k z|^2\Bigr)
+\sum_{k=0}^1\sum_{i,j=1}^n
\Bigl(|\p_t^k\p_i\p_j y|^2+|\p_t^k\p_i\p_j z|^2\Bigr)
\Biggr)\notag\\
&\qquad\qquad
+s\sum_{k=0}^1\Bigl(|\nabla\p_t^k y|^2+|\nabla\p_t^k z|^2\Bigr)
+s^3\sum_{k=0}^1\Bigl(|\p_t^k y|^2+|\p_t^k z|^2\Bigr)
\Biggr\}\,dx\,dt
\notag\\
&\quad\le
\int_{Q_{I_0}}\Biggl\{\frac{1}{s}\Biggl(
\sum_{k=1}^2\Bigl(|\p_t^k y|^2+|\p_t^k z|^2\Bigr)
+\sum_{k=0}^1\sum_{i,j=1}^n
\Bigl(|\p_t^k\p_i\p_j y|^2+|\p_t^k\p_i\p_j z|^2\Bigr)
\Biggr)\notag\\
&\qquad\qquad
+s\sum_{k=0}^1\Bigl(|\nabla\p_t^k y|^2+|\nabla\p_t^k z|^2\Bigr)
+s^3\sum_{k=0}^1\Bigl(|\p_t^k y|^2+|\p_t^k z|^2\Bigr)
\Biggr\}e^{2s\psi}\,dx\,dt
\notag\\
&\quad\le
Ce^{Cs}\,\mathcal{B}+Ce^{2s\rho_1}\,\mathcal{N}
\label{eq:holder-carleman}
\end{align}
for all $s\ge s_0$.

\medskip
\noindent\textbf{Step 4: H\"older interpolation and source recovery.}

Using the separation $\rho_2>\rho_1$ from \eqref{eq:rho-separation},
dividing \eqref{eq:holder-carleman} by $e^{2s\rho_2}$,
and minimising the right-hand side over $s>0$, 
we find $\theta\in(0,1)$ such that
\begin{align}
&\int_{\Om_0\times(t_0-r,t_0+r)}
\Biggl[\sum_{k=1}^2\Bigl(|\p_t^k y|^2+|\p_t^k z|^2\Bigr)
+\sum_{k=0}^1\Bigl(\sum_{i,j=1}^n
(|\p_t^k\p_i\p_j y|^2+|\p_t^k\p_i\p_j z|^2)\notag\\
&\qquad\qquad
+|\nabla\p_t^k y|^2+|\nabla\p_t^k z|^2
+|\p_t^k y|^2+|\p_t^k z|^2\Bigr)\Biggr]\,dx\,dt
\le C\,\mathcal{B}^{\theta}.
\label{eq:holder-interp}
\end{align}
Such a minimising argument is rather commonly applied and see for 
example \cite[Proof of Theorem~3.3.1]{Y25}).

The Sobolev embedding implies
$$
H^1(t_0-r,t_0+r;L^2(\Om_0))
\hookrightarrow C([t_0-r,t_0+r];L^2(\Om_0)),
$$
that is, 
$$
\left\|\partial_t y\left(\cdot, t_0\right)\right\|_{L^2\left(\Omega_0\right)} \leq C\left\|\partial_t y\right\|_{H^1\left(t_0-r, t_0+r ; L^2\left(\Omega_0\right)\right)}
$$
and
$$
\left\|\Delta y\left(\cdot, t_0\right)\right\|_{L^2\left(\Omega_0\right)} \leq C\|\Delta y\|_{H^1\left(t_0-r, t_0+r ; L^2\left(\Omega_0\right)\right)}.
$$
The left-hand side of \eqref{eq:holder-interp} controls
$\norm{\p_t y(\cdot,t_0)}_{L^2(\Om_0)}$ and
$\norm{\Lap y(\cdot,t_0)}_{L^2(\Om_0)}$.
Evaluating the first equation of \eqref{eq:diffsystem} at $t=t_0$ gives
\[
F(x)=\p_t y(x,t_0)-d_1\Lap y(x,t_0)-p_1(x,t_0)y(x,t_0)
-p_2(x,t_0)z(x,t_0),
\]
and hence
\[
\norm{F}_{L^2(\Om_0)}\le C\Bigl(
\norm{\p_t y(\cdot,t_0)}_{L^2(\Om_0)}
+\norm{\Lap y(\cdot,t_0)}_{L^2(\Om_0)}
+\norm{y(\cdot,t_0)}_{L^2(\Om_0)}
+\norm{z(\cdot,t_0)}_{L^2(\Om_0)}\Bigr).
\]
Each term on the right is controlled by \eqref{eq:holder-interp} and the
data norm $\mathcal{B}$, yielding
$$\norm{f-\tilde f}_{L^2(\Om_0)}\le C\,\mathcal{B}^\theta.$$
\end{proof}

The application of the Carleman estimate to only three functions $(y_1,z,z_1)$ rather than four $(y,z,y_1,z_1)$ is a technical advantage specific to the structure of the system: since the unknown source $F(x)$ appears only in the equation for $y$, differentiating in $t$ eliminates $F$ from the equation for $y_1$, and the remaining occurrence of $y$ can be handled through the integral relation $y=\int_{t_0}^t y_1\,d\eta+y_0$. This is why the data $v(\cdot,t_0)|_\Om$ is not needed in the present
formulation.
If the second equation in \eqref{eq:diffsystem} also contained an unknown source in $x$, then the full set of four functions would be required, necessitating the additional snapshot $z(\cdot,t_0)$.

\section{Proofs of Corollaries~\ref{cor:red2}--\ref{cor:red4}}
\label{app:reduction}

We provide detailed proofs of 
Corollaries~\ref{cor:red2}--\ref{cor:red4}. Recall that $I:= (t_1, \, t_2)$ with $0<t_1<t_0<t_2 < T$. We begin with the following lemma.
\begin{lemma}
\label{lem:pos}
Let $(u,v)$ be a solution to \eqref{eq:forward} - \eqref{eq:ic} with $a_1, a_2 \in C(\ooo{Q})$.
\\
(i) Let $v \in C(\ooo\OOO \times \ooo{I}) \cap C^2(\OOO \times I)$ and 
$u \in C(\ooo\OOO \times \ooo{I})$.
If \eqref{eq:5} and \eqref{eq:6} hold, then there exists a constant $m_0 > 0$ such that 
$v(x,t) \ge m_0$ for $x \in \ooo{\omega}$ and $t \in \ooo{I}$.
\\
(ii) Let $u \in C(\ooo\OOO \times \ooo{I}) \cap C^2(\OOO \times I)$,
$v \in C(\ooo\OOO \times \ooo{I})$, and let $f \ge 0$ in $\OOO$.
If \eqref{eq:7} and \eqref{eq:8} hold, then there exists a constant $m_0 > 0$ such that 
$u(x,t) \ge m_0$ for $x \in \ooo{\omega}$ and $t \in \ooo{I}$.
\end{lemma}
{\bf Proof of Lemma 5.1.}
\\
{\bf Proof of (i).}
We rewrite the second equation in \eqref{eq:forward} as 
$$
\ppp_tv = d_2\Delta v + b(x,t)v \quad \text{in } Q:= \OOO\times (0,T),
$$
where $b:= uv - a_2$ in $Q$. We note that $b \in C(\ooo{Q})$.
By  \eqref{eq:5}, the weak maximum principle (e.g., Renardy and Rogers \cite{RR})
implies that $v\ge 0$ on $\ooo{Q}$.
Assume that there exist $x_*\in \ooo{\omega} \subset \OOO$ and 
$t_* \in [t_1,t_2]$ with $t_1 > 0$ such that 
$v(x_*,t_*) = 0$.
This means that $v$ attains the minimum $0$ over $\ooo{Q}$ 
at an interior point $(x_*,t_*) \in Q$. Hence, the strong maximum principle
(e.g., Theorem 4.26 (p.122) in \cite{RR}) yields that 
$v = 0$ on $\ooo{\OOO} \times [0,t_*]$, which is a contradiction with 
$v_0 = v(\cdot,0) \not\equiv 0$ in $\OOO$.
Therefore, $\min_{(x,t) \in \ooo{\omega} \times \ooo{I}} v(x,t) > 0$.
\\
\\
{\bf Proof of (ii).}
We can rewrite the first equation in \eqref{eq:forward} as 
$$
\ppp_tu = d_1\Delta u + c(x,t)u + f \quad \text{in } Q:= \OOO\times (0,T),
$$
where $c:= -v^2 - a_1$ in $Q$. Similarly, by \eqref{eq:7} and $f \ge 0$ in $Q$,
the weak maximum principle implies that $u\ge 0$ on $\ooo{Q}$.
We note $-\ppp_tu + d_1\Delta u + cu = -f \le 0$ in $Q$.
Assume that there exist $x_*\in \ooo{\omega} \subset \OOO$ and 
$t_* \in [t_1,t_2]$ with $t_1 > 0$ such that 
$u(x_*,t_*) = 0$.
This means that $u$ attains the minimum $0$ over $\ooo{Q}$ 
at an interior point $(x_*,t_*) \in Q$. Hence, again 
the strong maximum principle yields that 
$u = 0$ on $\ooo{\OOO} \times [0,t_*]$, which is a contradiction with 
$u_0 = u(\cdot,0) \not\equiv 0$ in $\OOO$.
Thus the proof of (ii) is complete.
 $\blacksquare$ 

Each proof shows that the corresponding reduced measurement 
$\mathcal{M}_j$ determines both $u$ and $v$ in 
$\omega\times I$, thereby recovering the reference measurement 
$\mathcal{M}_1$.
The uniqueness of the source $f$ follows from 
Theorem~\ref{thm:reference}.

\begin{proof}[Proof of Corollary~\ref{cor:red2}]
The identity $\mathcal{M}_2(u,v,f) = \mathcal{M}_2(\widetilde u,\widetilde v,
\widetilde f)$ means, by definition, that
\[
u(\cdot,t_0) = \widetilde u(\cdot,t_0) \quad\text{in }\Omega, \qquad
v = \widetilde v \quad\text{in }\omega\times I.
\]
In $\omega\times I$, the pairs $(u,v)$ and $(\widetilde u,\widetilde v)$ 
satisfy
\[
\partial_t v=d_2\Delta v+uv^2-a_2v,
\qquad
\partial_t \widetilde v=d_2\Delta \widetilde v+\widetilde u\,\widetilde v^2
-a_2\widetilde v.
\]
Subtracting the two equations and using $v = \widetilde v$ in 
$\omega\times I$, 
we obtain
\[
(u-\tilde u)\,v^2=0
\qquad \text{in } \omega\times I.
\]
By Lemma \ref{lem:pos}(i), applied to $(u,v)$, 
there exists $m_0>0$ such that
\[
v(x,t)\ge m_0
\qquad \text{for all } (x,t)\in \omega\times I.
\]
Hence \(v^2>0\) in \(\omega\times I\), and therefore $ u = \widetilde u$
in $\omega\times I$.
Thus the full reference measurement \(\mathcal M_1\) is recovered.
Theorem~\ref{thm:reference} now yields
\[
f=\widetilde f \qquad \text{in } \Omega.
\]
\end{proof}

\begin{proof}[Proof of Corollary~\ref{cor:red3}]
The identity $\mathcal{M}_3(u,v,f)=\mathcal{M}_3(\widetilde u,\widetilde v,
\widetilde f)$ means, by definition, that
\[
u(\cdot,t_0)=\widetilde u(\cdot,t_0) \quad\text{in }\Omega, \qquad
u=\widetilde u \quad\text{in }\omega\times I, \qquad
f=\widetilde f \quad\text{in }\omega.
\]
In $\omega\times I$, the first equations for $(u,v)$ and $(\tilde u,\tilde v)$ give
\[
\partial_t u=d_1\Delta u-uv^2-a_1u+f,
\qquad
\partial_t \widetilde u=d_1\Delta \widetilde u
-\widetilde u\,\widetilde v^2-a_1\widetilde u+\widetilde f.
\]
Subtracting the two equations and using
\(u=\widetilde u\) in \(\omega\times I\) and \(f=\widetilde f\) in \(\omega\),
we obtain
\[
u\,(v^2-\widetilde v^2)=0
\qquad \text{in } \omega\times I.
\]
By Lemma~\ref{lem:pos}(ii), applied to $(u,v)$, there exists $m_0>0$ such that
\[
u(x,t)\ge m_0
\qquad \text{for all } (x,t)\in \omega\times I.
\]
Hence $ v^2=\widetilde v^2$ in $\omega\times I$.
Moreover, by Lemma~\ref{lem:pos}(i), applied to both $(u,v)$ and
$(\widetilde u,\widetilde v)$, one has
\[
v\ge 0,
\qquad
\widetilde v\ge 0
\qquad \text{in } \omega\times I.
\]
Therefore $ v=\widetilde v$ in $\omega\times I$.
Thus the reference measurement \(\mathcal M_1\) is recovered, and
Theorem~\ref{thm:reference} implies
\[
f=\widetilde f \qquad \text{in } \Omega.
\]
\end{proof}

\begin{proof}[Proof of Corollary~\ref{cor:red4}]
Let $y:=u-\widetilde u$, $z:=v-\widetilde v$, and $F:=f-\widetilde f$.
The identity $\mathcal{M}_4(u,v,f)=\mathcal{M}_4(\widetilde u,
\widetilde v,\widetilde f)$
means, by definition, that
\[
y(\cdot,t_0)=0 \quad\text{in }\Omega, \qquad
z(\cdot,t_0)=0 \quad\text{in }\Omega, \qquad
y=0 \quad\text{in }\omega\times I.
\]
Since $y=0$ in $\omega\times I$, one has
$\partial_t y=0$ and $\Delta y=0$ in $\omega\times I$.
Restricting the first equation of the difference system \eqref{eq:diffsystem} to \(\omega\times I\), we obtain
\[
0=\partial_t y-d_1\Delta y-p_1y=p_2z+F
\qquad \text{in } \omega\times I,
\]
where $p_2=-\widetilde u\,(v+\widetilde v)$ is defined by \eqref{eq:pcoeff}.
Evaluating this identity at \(t=t_0\) and using \(z(\cdot,t_0)=0\) 
in \(\Omega\),
we find that $ F(x)=0 \quad \text{in } \omega.$
Substituting this back into \(p_2z+F=0\), we obtain \[ p_2\,z=0 \quad 
\text{in } \omega\times I.\]
By Lemma~\ref{lem:pos}(ii), applied to both $(u,v)$ and 
$(\widetilde u,\widetilde v)$,
there exists \(m_1>0\) such that
\[
u(x,t)\ge m_1,
\qquad
\widetilde u(x,t)\ge m_1
\qquad \text{for all } (x,t)\in \omega\times I.
\]
By Lemma~\ref{lem:pos}(i), applied to both $(u,v)$ and 
$(\widetilde u,\widetilde v)$,
there exists \(m_2>0\) such that
\[
v(x,t)\ge m_2,
\qquad
\widetilde v(x,t)\ge m_2
\qquad \text{for all } (x,t)\in \omega\times I.
\]
Hence $|p_2(x,t)|
= \widetilde u(x,t)\,(v(x,t)+\widetilde v(x,t))
\ge 2m_1m_2>0$ in $\omega\times I$.
Therefore \(z=0\) on \(\omega\times I\), that is,
$v=\widetilde v$ on $\omega\times I$.
Since \(y=0\) in \(\omega\times I\), the reference measurement \(\mathcal M_1\)
is recovered.
Theorem~\ref{thm:reference} now yields
\[
f=\widetilde f
\qquad \text{in } \Omega.
\]
\end{proof}


\section{Conclusion}\label{sec:conclusion}

We have established stability estimates for the inverse source problem in a coupled reaction--diffusion system, namely the Klausmeier-Gray-Scott model arising in an ecosystem model. This system is characterised by the quadratic nonlinearity term. Our results can be summarised as follows.

Theorem~\ref{thm:reference} provides a Lipschitz stability estimate for the spatially dependent source $f$ from measurement data consisting of a
full-domain snapshot $u(\cdot,t_0)|_\Om$ and interior observation data 
of both components in $\omega\times I$.
Theorem~\ref{thm:holder} provides a H\"older stability estimate on any interior subdomain $\Om_0$ of $\Om$ from the same type of data.
Corollaries~\ref{cor:red2}--\ref{cor:red4} show that three reduced measurement configurations determine the same source under proper positivity assumptions on the observation region.

Data reduction in our inverse problem is possible because the algebraic structure of the nonlinear coupling term $uv^2$ compensates for missing observations.
This is a consequence of the specific form of the quadratic nonlinearity: the quadratic dependence on $v$ and the linear dependence on $u$ allow
algebraic recovery of one component from the other on the observation region, provided neither component vanishes there.

Several directions for further investigation are suggested by the present work. 
First, it is of interest to study the case where unknown source terms appear in both equations simultaneously, which requires the additional snapshot $v(\cdot,t_0)|_\Om$ and a modified Carleman argument.
Second, the extension to systems with advection (as in the original Klausmeier model on hillslopes, cf.\ \cite{klausmeier1999regular}) requires Carleman estimates for parabolic-transport equations.
Third, the development of efficient numerical algorithms for source reconstruction, building on the stability estimates established here,
is an important practical question.
Finally, it remains an interesting problem whether the nondegeneracy conditions \eqref{eq:5} - \eqref{eq:8} can be weakened or removed.

We also note that all results in this paper extend, with straightforward modifications, to the case where the source term $f(x)$ in the first equation of \eqref{eq:forward} is replaced by $R(x,t)\,f(x)$, provided $R$ is a known function satisfying $R(\cdot,t_0)>0$ on $\overline{\Om}$. Indeed, the source identity at $t=t_0$ in the proof of Theorem~\ref{thm:reference} reads 
$R(x,t_0)\,F(x)=\partial_t y(x,t_0)-d_1\Delta y(x,t_0)-p_1(x,t_0)y(x,t_0)
- p_2(x,t_0)z(x,t_0)$, and the positivity of $R(\cdot,t_0)$ allows 
the recovery of $F=f-\widetilde f$.

\section*{Acknowledgments}
This work was partly supported by the National Key Research and Development  Program of China (No.~2024YFA1012401) and the National Natural Science Foundation of China (No.~12241103). Masahiro Yamamoto was partly supported by Grant-in-Aid for Challenging Research (Pioneering) 21K18142 of Japan Society for the Promotion of Science.


\end{document}